\documentclass{amsart}
\usepackage{amsfonts}
\usepackage{color}

\usepackage{graphicx}

\newtheorem{thm}{Theorem}[section]
\newtheorem{cor}[thm]{Corollary}
\newtheorem{lema}[thm]{Lemma}

\theoremstyle{definition}

\theoremstyle{remark}
\newtheorem{rem}[thm]{Remark}
\theoremstyle{example}

\numberwithin{equation}{section}
\newcommand{\R}{\mathbb R}

\newcommand{\N}{\mathbb N}
\newcommand{\Z}{\mathbb Z}

\newcommand{\K}{\mathcal{K}}

\newcommand{\C}{\mathcal{C}}

\newcommand{\ve}{\varepsilon}

\newcommand{\lam}{\lambda}

\newcommand{\cd}{\rightharpoonup}
\newcommand{\cde}{\stackrel{*}{\rightharpoonup}}

\begin{document}
\title{Homogenization of Fu{\v{c}}{\'{\i}}k eigencurves}
\author[J. Fern\'andez Bonder, J.P. Pinasco and A.M. Salort]{Juli\'an Fern\'andez Bonder, Juan Pablo Pinasco and Ariel Martin Salort}

\address{Departamento de Matem\'atica, FCEyN - Universidad de Buenos Aires and
\hfill\break \indent IMAS - CONICET
\hfill\break \indent Ciudad Universitaria, Pabell\'on I (1428) Av. Cantilo s/n. \hfill\break \indent Buenos Aires, Argentina.}

\email[J. Fern\'andez Bonder]{jfbonder@dm.uba.ar}
\urladdr{http://mate.dm.uba.ar/~jfbonder}

\email[J.P. Pinasco]{jpinasco@dm.uba.ar}
\urladdr{http://mate.dm.uba.ar/~jpinasco}

\email[A.M. Salort]{asalort@dm.uba.ar}
\urladdr{http://mate.dm.uba.ar/~asalort}



\begin{abstract}
In this work we study the convergence of an homogenization problem for half-eigenvalues
and Fu{\v{c}}{\'{\i}}k eigencurves. We provide quantitative bounds on the rate of
convergence of the curves for periodic homogenization problems.
\end{abstract}

\maketitle
\section{Introduction}

Given a bounded interval $(0,\ell)\subset \R$, we are interested in the asymptotic behavior, as $\ve\to 0$, of the spectrum of the following family of asymmetric elliptic problems
\begin{align} \label{ecu}\tag{$P_\ve$}
\begin{cases}
-u''=   \alpha \, m_\ve(x)\, u^+ - \beta \,  n_\ve(x) \, u^-, \qquad x\in (0,\ell)\\
u(0)=u(\ell)=0,
\end{cases}
\end{align}
where $(\alpha,\beta) \in \R^2_+$, and the functions $m_\ve$, $n_\ve\in L^\infty([0,\ell])$ are positive and uniformly bounded between two positive constants,  
\begin{equation}\label{cotas.pesos}
0<a\leq m_\ve(x), n_\ve(x) \leq b < \infty.
\end{equation}
As usual, given a function $u$ we denote by $u^\pm = \max\{0,\pm u\}$ the positive and negative parts of $u$.

Here we assume that  there exist functions $m_0$, $n_0$  such that
$$
m_\ve \stackrel{*}{\cd} m_0 \quad \text{and}\quad n_\ve\stackrel{*}{\cd} n_0 \quad \text{weakly* in } L^\infty([0,\ell]).
$$
It is well-known that in the case of periodic homogenization, where
$$
m_\ve(x) = m(\tfrac{x}{\ve}), \qquad n_\ve(x) = n(\tfrac{x}{\ve})
$$
for some $\ell-$periodic functions $m, n\in L^\infty(\R)$, we have that $m_\ve \stackrel{*}{\cd} m_0$ and $n_\ve\stackrel{*}{\cd} n_0$ as $\ve\to 0$, where
$$
m_0 = \bar m :=  \frac{1}{\ell }\int_0^\ell m(x)\, dx \quad \text{and}\quad n_0 = \bar n :=  \frac{1}{\ell }\int_0^\ell n(x)\, dx.
$$

We will show that the following limit equation is obtained,
\begin{align} \label{ecu_lim}\tag{$P_0$}
\begin{cases}
-u''=    \alpha \,  m_0 u^+ - \beta \, n_0 u^-  \qquad x\in (0,\ell)\\
u(0)=u(\ell)=0,
\end{cases}
\end{align}
in the sense that, from any  sequence of weak solutions $\{(\alpha_{\ve_j}, \beta_{\ve_j}, u_{\ve_j})\}_{j\ge 1}$ of \eqref{ecu},  with $\ve_j\to 0$, uniformly bounded in $\R^2\times H_0^1([0,\ell])$, we can extract a strongly convergent subsequence in $\R^2\times L^2([0,\ell])$, weakly convergent in $\R^2\times H_0^1([0,\ell])$, and the limit is a weak solution of equation \eqref{ecu_lim}. Here, by a weak solution of \eqref{ecu} with $\ve\ge 0$, we understand a pair $(\alpha, \beta)\in \R^2_+$ and $u \in H_0^1([0,\ell])$ satisfying
\begin{equation}\label{ecu.debil}
\int_0^\ell u'v'\, dx=  \int_0^\ell  \left(\alpha \, m_\ve u^+ - \beta \, n_\ve u^-\right)v\, dx,
\end{equation}
for any $v\in H_0^1([0,\ell])$.

In order to study the convergence of the spectra of the problems \eqref{ecu} to the spectrum of \eqref{ecu_lim}, let us recall some known facts about the structure of the so-called Fu{\v{c}}{\'{\i}}k spectrum, introduced in the '70s by Dancer and Fu{\v{c}}{\'{\i}}k, see \cite{DAN, FUCIK-libro}. For any $\ve \ge 0$ fixed, let us denote by 
\begin{equation}\label{fucik.spectrum}
\Sigma_\ve = \Sigma(m_\ve, n_\ve) := \{(\alpha,\beta)\subset \R^2\colon \text{there exists a nontrivial solution of \eqref{ecu}}\},
\end{equation}
the Fu{\v{c}}{\'{\i}}k spectrum of problem \eqref{ecu}. In the above mentioned references, it is proved that the spectrum $\Sigma_\ve$ has the structure
$$
\Sigma_\ve = \mathcal{C}_{0,\ve}^\pm \cup \bigcup_{k\in \N } \mathcal{C}_{k,\ve}^\pm, 
$$
where each $\mathcal{C}_{k,\ve}^\pm$ is a curve in $\R^2$ for any $k\ge 0$. The curves $\mathcal{C}_{0,\ve}^\pm$ are called the trivial curves and are given by $\mathcal{C}_{0,\ve}^+=\{\lam_1^{m_\ve}\} \times \R$, $\mathcal{C}_{0,\ve}^- = \R \times \{\lam_1^{n_\ve} \}$, where $\lam_k^r$ denotes the $k-$th eigenvalue of the Dirichlet laplacian in $(0,\ell)$ with weight $r\in L^\infty([0,\ell])$, namely
\begin{equation}\label{ecu2}\begin{cases}
-u''= \lam r(x) u, & \qquad x\in (0,\ell)\\
u(0)=u(\ell)=0.
\end{cases}
\end{equation}
Observe that any eigenfunction associated to $\lambda_1^r$ has constant sign.

The curves $\mathcal{C}_{k,\ve}^+$ (resp. $\mathcal{C}_{k,\ve}^-$) with $k\ge 1$ correspond to nontrivial solutions having $k$ internal zeros and positive (resp. negative) slope at the origin. 

We have two curves for every $k\in \N$. In the constant coefficient case, for $k$ even, both curves coincide but this is not true for general weights. 

The curves $\mathcal{C}_{k,\ve}^\pm$ are not known explicitly
 for general weights $m_\ve, n_\ve$, and only its asymptotic behavior as $\alpha\to \infty$ (or $\beta\to \infty$) is known, see \cite{PiS, RYN}.

The study of homogenization problems for asymmetric eigenvalues is not well understood nowadays. We cite the paper \cite{Malik} of Malik where the homogenization problem for a model of suspension bridges was studied. In that work the author studies a model where the  cable resist the expansion but does not resist compression. More recently, in \cite{LY10}, Li and Yan studied the continuity of the eigenvalues $\lam(a_n,b_n)$ of the problem
$$
-(|u'|^{p-2}u')' = \lam |u|^{p-2} u + a_n(x) |u^+|^{p-2} u^+ -b_n(x) |u^-|^{p-2} u^-, \qquad x\in (0,\ell)
$$
with homogeneous boundary conditions
$$
c_{11}u(0)+c_{12}u'(0)=0=c_{21}u(\ell)+c_{22}u'(\ell),
$$
and the convergence to the eigenvalues of
$$
-(|u'|^{p-2}u')' = \lam |u|^{p-2} u + a(x) |u^+|^{p-2} u^+ -b(x) |u^-|^{p-2} u^-, \qquad x\in (0,\ell)
$$
with the same boundary conditions, where $a_n \cd a$ and $b_n\cd b$ weakly in $L^\gamma([0,\ell])$ for $1\le \gamma< \infty$.

Also, the behavior as $\ve\to 0$ of the first nontrivial curve in the Fu{\v{c}}{\'{\i}}k spectrum for the $p-$Laplace operator in $\R^n$ for $n\geq 1$ was obtained by the third author in \cite{S14}.

On the other hand, the homogenization of spectral problems in the symmetric case have been widely studied in both the linear and quasilinear cases. See for example \cite{Con, zuazua, ChP07, FBPS1, Kenig, Kes1, Kes2, Ol}  and the references therein. 

In this work we prove the convergence of the eigenvalues of problem \eqref{ecu} to the ones of problem \eqref{ecu_lim}. Moreover, in the case of periodic homogenization  we obtain the rate of convergence whenever we restrict $\Sigma_\ve$ and $\Sigma_0$ to a line through the origin, and we give explicit bounds depending on $\ve$, $k$, and the slope of the line. 

Since the constant degenerates when the line approaches the axis, it is convenient to denote, for any $0<t<1$ by  $\K_t$ a symmetric cone in the first quadrant defined by
\begin{equation}\label{eq.cone}
\K_t := \{(\alpha,\beta)\in \R_+\times \R_+\colon t\alpha\le \beta\le t^{-1}\alpha\}.
\end{equation}
Our main results are the following:
\begin{thm}[General convergence] \label{main} 
Let $\{m_\ve\}_{\ve>0}$ and $\{n_\ve\}_{\ve>0}$ be two families of weights satisfying \eqref{cotas.pesos} such that 
$$
m_\ve \stackrel{*}{\cd} m_0 \quad\text{and}\quad n_\ve\stackrel{*}{\cd}n_0
$$
weakly* in $L^\infty([0,\ell])$ and let $\Sigma_\ve$ be the associated Fu{\v{c}}{\'{\i}}k spectrum defined in \eqref{fucik.spectrum}.

Let $(\alpha_{k,\ve},\beta_{k,\ve}) \in \mathcal{C}_{k,\ve}\cap \K_t\subset \Sigma_\ve$. Then, $\{(\alpha_{k,\ve},\beta_{k,\ve})\}_{\ve>0}$ is bounded in $\R^2$ and if  $(\alpha_{k,0},\beta_{k,0})$ is any accumulation point of $\{(\alpha_{k,\ve},\beta_{k,\ve})\}_{\ve>0}$, then $(\alpha_{k,0},\beta_{k,0}) \in \C_{k,0}\cap\K_t\subset \Sigma_0$.

Moreover, if $(\alpha_{k,\ve},\beta_{k,\ve}) \in \mathcal{C}_{k,\ve}^+\cap \K_t$, then $(\alpha_{k,0},\beta_{k,0})\in \C_{k,0}^+$ and analogous result for $\C_{k,\ve}^-$.

Finally, if $u_{\ve}\in H^1_0([0,\ell])$ is an eigenfunction of \eqref{ecu} associated to $(\alpha_{\ve},\beta_{\ve})$ normalized such that $\|u_{\ve}\|_2=1$, then, there exists $u_0\in H^1_0([0,\ell])$ and a sequence $\ve_{j}\downarrow 0$ such that $u_{\ve_{j}}\cd u_0$ and $u_0$ is an eigenfunction of \eqref{ecu_lim} associated to $(\alpha_0,\beta_0)$.
\end{thm}

In the case of periodic homogenization one can do better and obtain an order of convergence. In order to do this one needs to select a point on the curve of the spectrum $\Sigma_\ve$ and follow that point as $\ve\downarrow 0$. This is done in the following way: given $t>0$, there exists a unique $\lambda_{k,t,\ve}^\pm$ such that $(\lambda_{k,t,\ve}^\pm, t\lambda_{k,t,\ve}^\pm)\in \C_{k,\ve}^\pm$. Moreover
$$
\C_{k,\ve}^\pm = \bigcup_{t>0} \{(\lambda_{k,t,\ve}^\pm, t\lambda_{k,t,\ve}^\pm)\}.
$$

\begin{thm}[Periodic homogenization] \label{main2} 
Assume that $m_\ve(x) = m(\tfrac{x}{\ve})$ and $n_\ve(x) = n(\tfrac{x}{\ve})$ for some $\ell-$periodic functions $m, n \in L^\infty(\R)$ satisfying \eqref{cotas.pesos}.

Then, we have the bound
$$
|\lambda_{k,t,\ve}^\pm - \lambda_{k,t,0}^\pm| \le C\left(\frac{k}{\ell}\right)^3\gamma(t)\ve,
$$
where $C$ depends only on the constants $a, b$ in \eqref{cotas.pesos} and $\gamma(t)=\max\{t^{-\frac32}, t^\frac12\}$.
\end{thm}

The order of convergence for homogenization of different eigenvalue problems were obtained in \cite{zuazua, FBPS, Kenig, S14}. Let us recall that  in \cite{zuazua, Kenig} the problem was linear, and asymptotic expansions were used. On the other hand, in \cite{FBPS, S14} the proofs
relayed on the variational structure of the problem. Here, there are no variational characterization of the higher curves of the Fu{\v{c}}{\'{\i}}k   spectrum, nor linear arguments available, so the proofs  are obtained by exploiting the nodal structure of the eigenfunctions.

\subsection*{Organization of the paper}
The paper is organized as follows:  In Section \S 2 we prove the general convergence result, Theorem \ref{main}, and in Section \S 3 we study the periodic oscillation case and prove Theorem \ref{main2}.

\section{A general convergence result}

In this section we prove our general convergence result, Theorem \ref{main}. We begin with an even more general, and therefore more vague, result on the convergence of Fu{\v{c}}{\'{\i}}k eigenvalues.

Throughout this section, we will use the notation $\lambda_1^{r,I}$ to denote the first eigenvalue of the Laplacian on the interval $I$ with weight function $r$ complemented with homogeneous Dirichlet boundary conditions. That is, $\lambda_1^{r,I}$ is the first eigenvalue of
$$
\begin{cases}
-u'' = \lambda r(x) u & \text{in } I\\
u=0 & \text{on }\partial I.
\end{cases}
$$
Let us recall, that if the weight $r(x)=constant = c$ then $\lambda_1^{r,I} = \lambda_1^{c,I} = \frac{\pi^2}{c|I|^2}$.

\begin{thm}\label{conv.alpha.beta}
Let $m_\ve$ and $n_\ve$ be two weight functions satisfying \eqref{cotas.pesos} and assume that $m_{\ve_j} \cde m_0$, $n_{\ve_j}\cde n_0$ weakly* in $L^\infty([0,\ell])$.  Let $\Sigma_\ve$ ($\ve\ge 0$) be the  Fu{\v{c}}{\'{\i}}k spectrum given by \eqref{fucik.spectrum}.    

If $(\alpha_{\ve_j},\beta_{\ve_j})\in \Sigma_{\ve_j}$ are such that $(\alpha_{\ve_j},\beta_{\ve_j})\to (\alpha_0,\beta_0)$ as $j\to\infty$, then $(\alpha_0,\beta_0)\in \Sigma_0$. Moreover, if $u_{\ve_j}\in H^1_0([0,\ell])$ is an eigenfunction of \eqref{ecu} associated to $(\alpha_{\ve_j},\beta_{\ve_j})$ normalized such that $\|u_{\ve_j}\|_2=1$, then, there exists $u_0\in H^1_0([0,\ell])$ and a subsequence $\ve_{j_i}\downarrow 0$ such that $u_{\ve_{j_i}}\cd u_0$ and $u_0$ is an eigenfunction of \eqref{ecu_lim} associated to $(\alpha_0,\beta_0)$.
\end{thm}

\begin{proof}
Let $u_{\ve_j}\in H^1_0([0,\ell])$ be an eigenfunction of \eqref{ecu} associated to $(\alpha_{\ve_j},\beta_{\ve_j})$ and normalized such that $\|u_{\ve_j}\|_2=1$.

Then, since $(\alpha_{\ve_j},\beta_{\ve_j})$ is bounded and since the weights $m_{\ve_j}$, $n_{\ve_j}$ are uniformly bounded, taking $v=u_{\ve_j}$ as a test function in \eqref{ecu.debil} we get
\begin{align*}
\int_0^\ell |u_{\ve_j}'|^2\, dx &= \alpha_{\ve_j} \int_0^\ell m_{\ve_j} (u_{\ve_j}^+)^2\, dx + \beta_{\ve_j} \int_0^\ell n_{\ve_j} (u_{\ve_j}^-)^2\, dx\\
&\le C \int_0^\ell (u_{\ve_j}^+)^2 + (u_{\ve_j}^-)^2\, dx = C \|u_{\ve_j}\|_2^2.
\end{align*}
Therefore, there exists a subsequence, that we still denoting by $\ve_j\downarrow 0$, and $u_0\in H^1_0([0,\ell])$ such that $u_{\ve_j}\cd u_0$ weakly in $H^1_0([0,\ell])$ and $u_{\ve_j} \to u_0$ uniformly in $[0,\ell]$. These facts automatically imply that $(u_{\ve_j}^\pm)^2 \to (u_0^\pm)^2$ strongly in $L^1([0,\ell])$.

So, we can pass to the limit in the weak form of the equation, \eqref{ecu.debil} to obtain
$$
\int_0^\ell u_0' v'\, dx = \alpha_0 \int_0^\ell m_0 u_0^+ v\, dx - \beta_0 \int_0^\ell u_0^- v\, dx,
$$
for every $v\in H^1_0([0,\ell])$. This finishes the proof.
\end{proof}

Let us now see that if we take a sequence $\{(\alpha_{k,\ve},\beta_{k,\ve})\}_{\ve>0}\subset \C_{k,\ve}$ with a fixed $k\in\N$, then the sequence of eigenvalues remains uniformly bounded as long as they are confined in a cone $\K_t$.

\begin{thm}\label{cota.alpha.beta}
Given $0<t<1$ let $\K_t$ be the cone defined in \eqref{eq.cone}.

Let $k\in \N$ be fixed and consider $(\alpha_{k,\ve},\beta_{k,\ve})\in \C_{k,\ve}\cap \K_t$.
Then, we have the bound
$$
\max\{\alpha_{k,\ve}, \beta_{k,\ve}\} \le t^{-1} \frac{\pi^2 k^2}{a\ell^2}.
$$
\end{thm}

\begin{proof}
Let $u_{k,\ve}\in H^1_0([0,\ell])$ be a eigenfunction of \eqref{ecu} associated to $(\alpha_{k,\ve}, \beta_{k,\ve})\in  \C_{k,\ve}\cap \K_t$. Then $u_{k,\ve}$ has exactly $k$ nodal domains. Therefore there exists at least one nodal domain, $I_\ve$, such that $|I_\ve|\ge \frac{\ell}{k}$.

Assume that $u_{k,\ve}>0$ in $I_\ve$ (the other case can be treated similarly). Therefore, $u_{k,\ve}$ is a weak solution of
$$
\begin{cases}
-u_{k,\ve}'' = \alpha_{k,\ve} m_\ve u_{k,\ve} & \text{in } I_\ve\\
u_{k,\ve} = 0 & \text{on }\partial I_\ve.
\end{cases}
$$
So, $\alpha_{k,\ve}= \lambda_1^{m_\ve, I_\ve}$. Now, by Sturm's comparison Theorem, we get
$$
\alpha_{k,\ve} = \lambda_1^{m_\ve,I_\ve} \le \lambda_1^{a,I_\ve} =  \frac{\pi^2}{a|I_\ve|^2} \le \frac{\pi^2 k^2}{a\ell^2}.
$$
Since $(\alpha_{k,\ve}, \beta_{k,\ve})\in \K_t$ it follows that,
$$
\beta_{k,\ve}\le t^{-1}\alpha_{k,\ve}.
$$
This completes the proof.
\end{proof}

Finally, let us see that the nodal domains of an eigenfunction $u_{k,\ve}$ of \eqref{ecu} associated to $(\alpha_{k,\ve}, \beta_{k,\ve})\in \C_{k,\ve}$  do not degenerate when we pass to the limit $\ve\downarrow 0$ if the eigenvalues $(\alpha_{k,\ve}, \beta_{k,\ve})$ are confined to a cone $\K_t$.

\begin{thm}\label{cota.nodal.domain}
With the same notations and assumptions of the previous theorem, let $(\alpha_{k,\ve},\beta_{k,\ve})\in \C_{k,\ve}\cap \K_t$ and let $u_{k,\ve}\in H^1_0([0,\ell])$ be an eigenfunction of \eqref{ecu} associated to $(\alpha_{k,\ve},\beta_{k,\ve})$. Then, every nodal domain $I_\ve\subset [0,\ell]$ of $u_{k,\ve}$ verifies the bound
$$
|I_\ve|\ge \frac{\ell}{k}\sqrt{t\frac{a}{b}}.
$$
Moreover if we denote by $J_\ve$ two consecutive nodal domains, we have the bound
$$
|J_\ve|\ge \frac{\ell}{k}\sqrt{\frac{a}{b}} (1+\sqrt{t}).
$$
\end{thm}

\begin{proof}
Assume that $u_{k,\ve}>0$ in $I_\ve$ (the other case is analogous). Arguing as in the proof of Theorem \ref{cota.alpha.beta}, we have that $\alpha_{k,\ve} = \lambda_1^{m_\ve,J_\ve}$. So, by Sturm's comparison Theorem,
$$
\alpha_{k,\ve} = \lambda_1^{m_\ve, J_\ve}\ge \lambda_1^{b,J_\ve} = \frac{\pi^2}{b|I_\ve|^2}.
$$
Now, using the bound for $\alpha_{k,\ve}$ given in Theorem \ref{cota.alpha.beta}, we deduce
$$
\frac{\pi^2}{b|I_\ve|^2} \le \alpha_{k,\ve}\le  t^{-1} \frac{\pi^2 k^2}{a\ell^2},
$$
and the result follows.

Let now $I^+_\ve$ and $I_\ve^-$ be two consecutive nodal domains, such that $u_{k,\ve}>0$ in $I_\ve^+$ and $u_{k,\ve}^-<0$ in $I_\ve^-$. We can assume, without loss of generality, that $\alpha_{k,\ve}\le \beta_{k,\ve}$. Then, from Theorem \ref{cota.alpha.beta}, we have that
$$
\alpha_{k,\ve}\le \frac{\pi^2 k^2}{a \ell^2} \quad \text{and}\quad \beta_{k,\ve}\le t^{-1}\frac{\pi^2 k^2}{a \ell^2}.
$$ 
Then, arguing as in the first part of the proof, we obtain that
$$
|I_\ve^+|\ge \frac{\ell}{k} \sqrt{\frac{a}{b}}\quad \text{and}\quad |I_\ve^-|\ge \frac{\ell}{k} \sqrt{t\frac{a}{b}}.
$$
The result follows observing that $|J_\ve| = |I_\ve^+| + |I_\ve^-|$.
\end{proof}

With the help of Theorems \ref{conv.alpha.beta}, \ref{cota.alpha.beta} and \ref{cota.nodal.domain}, the proof of Theorem \ref{main} follows easily.

\begin{proof}[Proof of Theorem \ref{main}]
It only remains to see that if $\C_{k,\ve}\ni (\alpha_{k,\ve_j}, \beta_{k,\ve_j})\to (\alpha_{k,0}, \beta_{k,0})$ as $j\to\infty$, then $(\alpha_{k,0}, \beta_{k,0})\in \C_{k,0}$. This will follow if we show that an associated eigenfunction $u_{k,0}$ of \eqref{ecu_lim} associated to $(\alpha_{k,0}, \beta_{k,0})$ has $k$ nodal domains.

But, from Theorem \ref{conv.alpha.beta}, we know that $u_{k,\ve_j}\cd u_{k,0}$ weakly in $H^1_0([0,\ell])$, where $u_{k,\ve_j}$ is an eigenfunction of \eqref{ecu} associated to $(\alpha_{k,\ve_j}, \beta_{k,\ve_j})$ and $u_{k,0}$ is an eigenfunction of \eqref{ecu_lim} associated to $(\alpha_{k,0}, \beta_{k,0})$. Therefore, we know that $u_{k,0}$ has only finitely many zeroes  and then from Theorem \ref{cota.nodal.domain} we deduce that $u_{k,0}$ has exactly $k$ nodal domains.

This completes the proof.
\end{proof}

\section{An alternative formulation}

In order to prove the convergence result for periodic homogenization, Theorem \ref{main2}, it is convenient to consider the following equivalent problem
\begin{align}  \label{ec.1}
	\begin{cases}
	   -u''= \lam  (m(x)u^+-  t n(x)u^-) \qquad x\in (0,\ell)\\
	   u(0)=u(\ell)=0,
	\end{cases}
\end{align}
where $t>0$ is a fixed value. The values of $\lam\in\R$ for which \eqref{ec.1} has a
non-trivial solution $u$ are called \emph{half-eigenvalues}, while the corresponding
solutions $u$ are called \emph{half-eigenfunctions}. Problem \eqref{ec.1} has a
positively-homogeneous jumping nonlinearity, and its spectrum is defined as the set
$$
\Sigma_t(m,n):=\{\lam\in \R \, : \, \eqref{ec.1} \mbox{ has non-trivial solution }u \}.
$$

The set $\Sigma_t(m,n)$ is divided into two subsets $\Sigma_t(m,n) = \Sigma_t^+(m,n)\cup \Sigma_t^-(m,n)$ as
$$	
\Sigma_t^+(m,n) := \{\lam\in \Sigma_t(m,n)\colon u_\lambda'(0)>0\},\ \Sigma_t^-(m,n) := \{\lam\in \Sigma_t(m,n)\colon u_\lambda'(0)<0\}
$$
where $u_\lambda$ is an eigenfunction of \eqref{ec.1} associated to $\lambda$.

It is shown in \cite{LY09} that for any $t>0$ both sets $\Sigma_t^{\pm}(m,n)$ consists in a sequence converging to $+\infty$. We denote this sequences by $\{\lambda_{k,t}^\pm\}_{k\in\N}$.

Observe that $\lambda_{1,t}^+ = \lambda_1^{m, [0,\ell]}$ and $\lambda_{1,t}^- = \lambda_1^{tn, [0,\ell]}$. Moreover, each eigenvalue has a unique associated eigenfunction normalized by $\pm u'(0)=1$ and the eigenfunction corresponding to $\lam_{k,t}^\pm$ has precisely $k$ nodal domains on $(0,\ell)$, and $k+1$ simple zeros in $[0, \ell]$. See \cite{LY09} for a proof of these facts.

We can rewrite problems \eqref{ecu} and \eqref{ecu_lim}  by taking $\lam_\ve = \alpha_\ve$
and $\beta_\ve = t \alpha_\ve$, and we obtain the following problems
\begin{align} \tag{$P_\ve^t$} \label{ec.1h}
\begin{cases}
-u''= \lam  (m_\ve(x) u^+ -  t n_\ve(x )u^-) & x\in (0,\ell)\\
u(0) = u(\ell)=0,
\end{cases}
\end{align}
for $\ve\ge 0$.

We denote the eigenvalues of \eqref{ec.1h} by $\{\lambda_{k,t,\ve}^\pm\}_{k\in\N}$.

Now, Theorem \ref{main} trivially implies the following
\begin{thm} \label{homog.t.fijo}
Let $\{m_\ve\}_{\ve>0}$ and $\{n_\ve\}_{\ve>0}$ be two families of weights in $L^\infty([0,\ell])$   satisfying \eqref{cotas.pesos}. Assume, moreover that $m_\ve\stackrel{*}{\cd} m_0$ and $n_\ve\stackrel{*}{\cd} n_0$ weakly* in $L^\infty([0,\ell])$ for some limit weights $m_0$ and $n_0$.
 
Let us denote by $\{\lam_{k,t,\ve}^\pm\}_{k\in\N}$ the eigenvalues of \eqref{ec.1h} for $\ve\ge 0$. Then
$$
\lim_{\ve \to 0}\lam_{k,t,\ve}^\pm = \lam_{k,t,0}^\pm.
$$
\end{thm}

Now we specialize to the periodic case, and obtain the following refinement.
\begin{thm} \label{homog.t.fijo.period}
In addition to the assumptions of Theorem \ref{homog.t.fijo}, assume that $m_\ve(x) = m(\tfrac{x}{\ve})$ and $n_\ve(x) = n(\tfrac{x}{\ve})$ for some $\ell-$periodic functions $m, n\in L^\infty(\R)$. Then, for $0<t<1$, 
$$
|\lam_{k,t,\ve}^\pm - \lam_{k,t,0}^\pm| \leq  C \left(\frac{k}{\ell}\right)^3 t^{-\frac32} \ve,
$$
where $C$ depends onlyt on $a, b$ in \eqref{cotas.pesos}.
\end{thm}

Observe that Theorem \ref{main2} follows directly from Theorem \ref{homog.t.fijo.period}. In fact, Theorem \ref{homog.t.fijo.period} is Theorem \ref{main2} in the case $0<t<1$ and the case where $t>1$ follows from this one by symmetry. To be precise, if $t>1$ and $u_\ve$ is an eigenfunction associated to $\lambda_{k,t,\ve}^\pm$, we denote $v_\ve=-u_\ve$ and so it verifies
$$
\begin{cases}
-v_\ve'' = t\lambda_{k,t,\ve}^\pm (n_\ve v_\ve^+ - t^{-1} m_\ve v_\ve^-) & \text{ in } (0,\ell)\\
v(0)=v(\ell)=0.
\end{cases}
$$

Therefore, from Theorem \ref{homog.t.fijo.period} we have the estimate
$$
|t\lambda_{k,t,\ve}^\pm - t\lambda_{k,t,0}^\pm|\le C\left(\frac{k}{\ell}\right)^3 t^{\frac32} \ve,
$$
and Theorem \ref{main2} follows directly from this former inequality.

For the proof of Theorem \ref{homog.t.fijo.period}, we need the order of convergence of the nodal domains of the eigenfunctions. We will perform this task in a series of lemmas.

\begin{lema}\label{dominios.nodales.epsilon}
Let $(\lambda^{\pm}_{k,t,\ve}, u_{k,t,\ve})$ be an eigenpair of \eqref{ec.1h}. We denote by $\{I_{j,\ve}^+\}_{j}\cup \{I_{i,\ve}^-\}_{i}$ the nodal domains of $u_{k,t,\ve}$, that is each $I_{l,\ve}^\pm$ is an open connected, pairwise disjoint intervals, such that
$$
[0,\ell] = \bigcup_{j} \overline{I_{j,\ve}^+} \cup \bigcup_{i} \overline{I_{i,\ve}^-}, 
$$
and $u_{k,t,\ve}>0$ on $I_{j,\ve}^+$, $u_{k,t,\ve}<0$ on $I_{i,\ve}^-$.

Then, $||I_{j,\ve}^+| - |I_{l,\ve}^+||<2\ve$ and $||I_{i,\ve}^-| - |I_{l,\ve}^-||<2\ve$
\end{lema}

\begin{proof}
We make the proof for the positive nodal domains $\{I_{j,\ve}^+\}_j$. The other one is analogous.

First, let $j_0$ be such that $|I_{j_0,\ve}^+|\le |I_{j,\ve}^+|$ for any $j$.

Assume that there exists $j$ such that $|I_{j,\ve}^+|>|I_{j_0,\ve}^+| + 2\ve$. Then, there exists an integer $h\in\Z$ such that $h\ve + I_{j_0,\ve}^+\subset I_{j,\ve}^+$.

Now, if we denote 
$$
v_\ve(x) = \begin{cases}
u_{k,t,\ve}(x-h\ve) & \text{if } x\in I_{j_0,\ve}^+ + h\ve\\
0 & \text{elsewhere},
\end{cases}
$$
then $v_\ve\in H^1_0(I_{j,\ve}^+)$, and so
\begin{align*}
\lambda_{k,t,\ve}^+ = \lambda_1^{m_\ve, I_{j,\ve}^+} & = \inf_{v\in H^1_0(I_{j,\ve}^+)} \frac{\int_{I_{j,\ve}^+} (v')^2\, dx}{\int_{I_{j,\ve}^+} m_\ve v^2\, dx}\\
&\le \frac{\int_{I_{j,\ve}^+} (v_\ve')^2\, dx}{\int_{I_{j,\ve}^+} m_\ve v_\ve^2\, dx}\\
&= \frac{\int_{I_{j_0,\ve}^+} (u_{k,t,\ve}')^2\, dx}{\int_{I_{j_0,\ve}^+} m_\ve u_{k,t,\ve}^2\, dx}\\
&= \lambda_{k,t,\ve}^+,
\end{align*}
where we have used the periodicity of the weight $m_\ve$. This shows that $v_\ve$ is an eigenfunction associated to $\lambda_1^{m_\ve, I_{j,\ve}^+}$, but this is a contradiction since $v_\ve$ vanishes in a set of positive measure.

The proof is complete.
\end{proof}

The following elementary lemma will be most useful.
\begin{lema}\label{lema.numeritos}
Let $M\in \R$ and $K\in\N$. Assume that for every $\ve>0$, there exists $\{a_i^\ve\}_{i=1}^K\subset\R$, such that
$$
\sum_{i=1}^K a_i^\ve = M \quad \text{and}\quad |a_i^\ve - a_j^\ve|< \ve, \text{ for every } 1\le i,j\le K.
$$
Then
$$
\left|a_i^\ve - \frac{M}{K}\right| < \ve, \text{ for every } 1\le i\le K.
$$
\end{lema}

\begin{proof}
Let $i_0 = i_0(\ve)\in \{1,\dots,K\}$ be such that $a_{i_0}^\ve\le a_i^\ve$ for every $1\le i\le K$. Then
$$
Ka_{i_0}^\ve \le \sum_{i=1}^K a_i^\ve = M,
$$
and so $a_{i_0}^\ve\le \frac{M}{K}$. Therefore, for any $1\le i\le K$, 
$$
a_i^\ve< a_{i_0}^\ve + \ve \le \frac{M}{K} + \ve.
$$
On the other hand, if we now take $i_1 = i_1(\ve) \in \{1,\dots,K\}$ such that $a_{i_1}^\ve\ge a_i^\ve$ for every $1\le i\le K$, then
$$
Ka_{i_i}^\ve \ge \sum_{i=1}^K a_i^\ve = M,
$$
thus $a_{i_1}^\ve\ge \frac{M}{K}$. Hence, for any $1\le i\le K$,
$$
a_i^\ve\ge a_{i_1}^\ve - \ve \ge \frac{M}{K} - \ve.
$$
The lemma is proved.
\end{proof}

Lemma \ref{dominios.nodales.epsilon} together with Lemma \ref{lema.numeritos} imply the following corollary.
\begin{cor}\label{conv.dominios.nodales}
Let $(\lambda^{\pm}_{k,t,\ve}, u_{k,t,\ve})$ be an eigenpair of \eqref{ec.1h}. We denote by $\{I_{j,\ve}^+\}_{j}\cup \{I_{i,\ve}^-\}_{i}$ the nodal domains of $u_{k,t,\ve}$. Then 
$$
\left| |I_{j,\ve}^+\cup I_{j,\ve}^-| - \frac{2\ell}{k}\right| \le 4\ve.
$$
\end{cor}

\begin{proof}
Assume first that $k$ is even. So $k=2n$ for some $n\in\N$. Then, the number of positive nodal domains    equals the number of negative nodal domains and both equal $n$. Therefore
$$
\ell = \sum_{j=1}^n |I_{j,\ve}^+| + \sum_{j=1}^n |I_{j,\ve}^-| = \sum_{j=1}^n |I_{j,\ve}^+\cup I_{j,\ve}^-|.
$$
Let us call $a_j^\ve = |I_{j,\ve}^+\cup I_{j,\ve}^-|$, and by Lemma \ref{dominios.nodales.epsilon} we have that $|a_j^\ve - a_i^\ve| < 4\ve$. Hence, we can invoke Lemma \ref{lema.numeritos} and conclude the desired result.

If now $k$ is odd, we consider the problem in $[-\ell,\ell]$ and extend $u_{k,t,\ve}$ by odd reflexion and so we end up with a even number of positive and negative nodal domains. We apply the first part of the proof and from that the result follows. We leave the details to the reader. 
\end{proof}

\begin{rem}\label{pesos.constantes}
Observe that, since $m_0$ and $n_0$ are constant functions, it holds that $|I_{j,0}^+\cup I_{j,0}^-| = \frac{2\ell}{k}$. See \cite{FUCIK-libro}.
\end{rem}

The other key ingredient in the proof of Theorem \ref{homog.t.fijo.period} is the following result due to \cite{S14}, recently improved in \cite{Salort}.
\begin{thm}[\cite{Salort}, Theorem 1.2]\label{conv.primera.curva}
Under the above assumptions and notations, it holds that
\begin{equation}\label{eq.primera.curva}
|\lam_{2,t,\ve}^\pm - \lam_{2,t,0}^\pm|\le C\ve t^{-\frac32},
\end{equation}
where $C$ is a constant depending only on the bounds $a, b$ in \eqref{cotas.pesos}.
\end{thm}

\begin{rem}
In \cite{S14} the obtained bound is slightly worse than \eqref{eq.primera.curva}. In fact, is was proved in \cite[Theorem 4.2]{S14} that
$$
|\lam_{2,t,\ve}^\pm - \lam_{2,t,0}^\pm|\le C'\ve t^{-2},
$$
with $C'$ depending also on $a, b$.
\end{rem}

With all of these preliminaries, we can now prove the main result of the section.
\begin{proof}[Proof of Theorem \ref{homog.t.fijo.period}]
Let $u_{k,t,\ve}$ be an eigenfunction of \eqref{ec.1h} associated to $\lam_{k,t,\ve}^+$, for $\ve\ge 0$. The case of $\lambda_{k,t,\ve}^-$ is completely analogous.

Let $J_\ve = I_{1,\ve}^+\cup I_{1,\ve}^-$ be the union of the first two nodal domains of $u_{k,t,\ve}$. Let us denote $J_\ve=(0,c_\ve)$. Observe that $u_{k,t,\ve}>0$ in $I_\ve$ for $\ve\ge 0$ and that, by Theorem \ref{cota.nodal.domain}, we have the bound
\begin{equation}\label{cota.ce}
c_\ve\ge \frac{\ell}{k}\sqrt{\frac{a}{b}}(1+\sqrt{t}).
\end{equation}

Arguing as in Theorem \ref{cota.alpha.beta}, we deduce that for any $\ve\ge 0$,
\begin{equation} \label{relac}
\lam_{k,t, \ve}^+ = \lambda_2^{m_\ve, tn_\ve, J_\ve},
\end{equation}
where $\lambda_2^{m_\ve, tn_\ve, J_\ve}$ is the second eigenvalue of \eqref{ec.1h} in the interval $J_\ve$.

Performing a change of variables is easy to see that 
\begin{equation}\label{cambio.variables}
c_\ve^2\lambda_2^{m_\ve, tn_\ve, J_\ve} = \lambda_2^{m_{\ve'}, tn_{\ve'},[0,1]},
\end{equation}
where $\ve' = \frac{\ve}{c_\ve}$. Observe that from \eqref{cota.ce} it follows that $\ve'\to 0$.

Now, using Theorem \ref{conv.primera.curva} we obtain that
\begin{equation} \label{ford}
\left| \lambda_2^{m_{\ve'}, tn_{\ve'}, [0,1]} - \lambda_2^{m_0, tn_0, [0,1]}\right| \leq C \ve' t^{-\frac32} \le C \frac{k}{\ell} t^{-\frac32} \ve.
\end{equation}
where $C$ depends on the constants $a, b$ in \eqref{cotas.pesos}.

Therefore, by\eqref{cota.ce},  \eqref{cambio.variables}  and \eqref{ford}, we find
\begin{align*}
|\lam_{k,t,\ve}^+ - \lam_{k,t,0}^+| &= |c_\ve^{-2}\lam_2^{m_{\ve'}, tn_{\ve'}, [0,1]} - c_0^{-2}\lam_2^{m_0, tn_0, [0,1]}| \\
&\le c_\ve^{-2} |\lam_2^{m_{\ve'}, tn_{\ve'}, [0,1]} - \lam_2^{m_0, tn_0,[0,1]}| + \lam_2^{m_0, tn_0, [0,1]} |c_\ve^{-2} - c_0^{-2}|\\
&\le C \left(\frac{k}{\ell}\right)^3 t^{-\frac32} \ve + \lam_2^{m_0, tn_0,[0,1]}  |c_\ve^{-2} - c_0^{-2}|.
\end{align*}
Now, from Corollary \ref{conv.dominios.nodales} and the subsequent remark, 
$$
|c_\ve^{-2} - c_0^{-2}|\le C \left(\frac{k}{\ell}\right)^3 \ve,
$$
with $C$ a universal constant. 

Finally, we observe that from Theorem \ref{cota.nodal.domain},
$$
\lam_2^{m_0, tn_0, [0,1]} = \lam_1^{m_0, I_{1,\ve}^+} = \frac{\pi^2}{m_0 |I_{1,\ve}^+|^2} \le C t,
$$
with $C$ depending on $a, b$ in \eqref{cotas.pesos}.
\end{proof}

\section*{Acknowledgements}

This paper was partially supported by Universidad de Buenos Aires under grant UBACyT 20020130100283BA and by ANPCyT under grant PICT 2012-0153. The authors are members of CONICET.

\bibliographystyle{amsplain}
\bibliography{Biblio}

\end{document}